\documentclass{amsproc}

\usepackage{amssymb}

\usepackage{graphicx}

\usepackage[cmtip,all]{xy}

\usepackage{amsmath,amssymb,enumerate}

\usepackage{amstext}
\usepackage{amsmath}
\usepackage{amsfonts}
\usepackage{latexsym}
\usepackage{ifthen}

\usepackage{xypic}
\xyoption{all}

\newtheorem{theorem}{Theorem}[section]
\newtheorem{lemma}[theorem]{Lemma}

\theoremstyle{definition}
\newtheorem{definition}[theorem]{Definition}
\newtheorem{example}[theorem]{Example}
\newtheorem{question}[theorem]{Question}
\newtheorem{proposition}[theorem]{Proposition}
\newtheorem{corollary}[theorem]{Corollary}
\newtheorem{conjecture}[theorem]{Conjecture}

\theoremstyle{remark}
\newtheorem{remark}[theorem]{Remark}
\newtheorem{conclusion}[theorem]{Conclusion}

\numberwithin{equation}{section}

\newcommand{\rk}{{\rm rk}}

\DeclareMathOperator*{\supp}{Supp}

\newcommand{\R}{\ensuremath{\mathbb{R}}}
\newcommand{\Q}{\ensuremath{\mathbb{Q}}}
\newcommand{\Z}{\ensuremath{\mathbb{Z}}}
\newcommand{\C}{\ensuremath{\mathbb{C}}}
\newcommand{\N}{\ensuremath{\mathbb{N}}}
\newcommand{\PP}{\ensuremath{\mathbb{P}}}

\newcommand{\merom}[3]{\ensuremath{#1:#2 \dashrightarrow #3}}

\newcommand{\holom}[3]{\ensuremath{#1:#2  \rightarrow #3}}
\newcommand{\fibre}[2]{\ensuremath{#1^{-1} (#2)}}
\newcommand\sA{{\mathcal A}}

\newcommand\sE{{\mathcal E}}

\newcommand\sF{{\mathcal F}}

\newcommand\sH{{\mathcal H}}

\newcommand\sO{{\mathcal O}}

\newcommand\bZ{{\mathbb Z}}

\newcommand\sD{{\mathcal D}}

\DeclareMathOperator*{\sing}{sing}

\DeclareMathOperator*{\nons}{nons}

\newcommand{\NEX}{\overline{\mbox{NE}}(X)}
\newcommand{\NAX}{\overline{\mbox{NA}}(X)}

\DeclareMathOperator*{\NS}{NS}

\title{Bimeromorphic geometry of K\"ahler threefolds} 

\author{Andreas H\"oring}

\address{Andreas H\"oring, Universit\'e C\^ote d'Azur, CNRS, LJAD, France}
\email{hoering@unice.fr}

\author{Thomas Peternell}
\address{Thomas Peternell, Mathematisches Institut, Universit\"at Bayreuth, 95440 Bayreuth, 
Germany}

\email{thomas.peternell@uni-bayreuth.de}


\date{\today}

\begin{document}

\begin{abstract}
We describe the recently established minimal model program for (non-algebraic) K\"ahler threefolds as well as the abundance theorem 
for these spaces. 
\end{abstract}

\maketitle


\section{Introduction}

Given a complex projective manifold $X$, the Minimal Model Program (MMP) predicts that either $X$ is covered by 
rational curves ($X$ is {\it uniruled}) or $X$ has a - slightly singular - birational {\it minimal} model $X'$ whose canonical divisor $K_{X'}$ is nef; and then
the abundance conjecture says that some multiple $mK_{X'}$ is spanned by global sections (so $X'$ is a {\it good minimal model}). 
The MMP also predicts how to achieve the birational model, namely by a sequence of divisorial contractions and flips. 
In dimension three, the MMP is completely established (cf. \cite{Uta92}, \cite{KM98} for surveys), in dimension four, the existence of minimal models is established (\cite{BCHM10}, \cite{Fuj04}, \cite{Fu05}), 
but abundance is wide open. In higher dimensions minimal models exists if $X$ is of general type \cite{BCHM10}; abundance not being an issue in this case. 

In this article we discuss the following natural

\begin{question}  Does the MMP work for general (non-algebraic) compact K\"ahler manifolds?  
\end{question} 

Although the basic methods used in minimal model theory all fail in the K\"ahler case,  there is no apparent reason why the MMP should not
hold in the K\"ahler category. And in fact, in recent papers \cite{HP16}, \cite{HP15} and \cite{CHP16},  the K\"ahler MMP was established in dimension three:

\begin{theorem} \label{theoremmain}
Let $X$ be a  normal $\Q$-factorial compact   K\"ahler threefold with  terminal singularities. Then there exists a MMP (i.e. a finite sequence of divisorial contractions and flips) 
$$
X \dashrightarrow X'
$$
such that the following holds:
\begin{itemize}
\item If $X$ is not covered by rational curves, then $X'$ is a good minimal model: there exists
a $m \in \N$ such that $m K_{X'}$ is generated by its global sections.
\item If $X$ is covered by rational curves, then $X'$ is a Mori fibre space: there exists
a fibration $\varphi: X' \rightarrow Y$ such that $-K_{X'}$ is $\varphi$-ample and $b_2(X') = b_2(Y)  + 1$.
\end{itemize}
\end{theorem}

Two main issues in the MMP are to construct ``contractions of extremal rays'' and to construct $K_X-$negative rational curves.
In the algebraic case, contractions are constructed by exhibiting nef line bundles of the form $K_X + H$ with $H$ ample. Suitable multiples of these line bundles are spanned and then the contraction is
given by the sections of the line bundle. Rational curves are constructed using reduction to char $p$. Both methods fundamentally
fail in the K\"ahler setting. One reason is of course, that there are only a few line bundles (and no ample line bundles)  on a general K\"ahler manifold, and also 
the Mori cone $\NEX$ is too small to be relevant. 
At least in dimension three, it is nevertheless possible to construct good minimal models. The aim of this 
article is to explain how this is achieved, including the necessary framework of K\"ahler geometry. To give a flavor of the difficulties arising in the 
K\"ahler setting, we mention the following 

\begin{conjecture} Let $X$ be a compact K\"ahler manifold such that the canonical class  is not nef (i.e, the first Chern class $c_1(K_X)$ is not in the closure of the
K\"ahler cone). Then there exists a rational curve $C $ such that $K_X \cdot C < 0.$ 
\end{conjecture} 

As we will see in the next section it is even not clear whether there is {\em any} curve $C$ such that $K_X \cdot C < 0$!

\section{Brief review of the algebraic case} \label{sectionalgebraic}

Let $X$ be a complex projective variety, and let $\NS(X)$ be its N\'eron-Severi group. We denote by $N^1(X):= \NS(X) \otimes \R$
the corresponding real vector space, and by $N_1(X)$ its dual, the vector space generated by classes of curves.
By definition a class $\alpha \in N^1(X)$ is nef if it is in the closure of the cone generated by ample divisors.
This somewhat abstract definition (which has the advantage of generalising to the K\"ahler case, cf. Section \ref{sectionkaehler})
can be translate in more geometric terms by Kleiman's criterion. We have
\begin{equation} \label{kleiman}
\alpha \in N^1(X) \mbox{ is nef } \Leftrightarrow \alpha \cdot C \geq 0 \qquad \forall \ C \subset X \mbox{ a curve}
\end{equation}
One of the cornerstones of the minimal model program is to give a much more precise description
of nefness when $\alpha$ is the class defined by the canonical divisor; for all definitions and basic results around the MMP 
we refer e.g. to \cite{KMM87} and \cite{KM98}. 

\begin{theorem} \label{theoremNE} 
Let $X$ be a normal $\Q$-factorial projective variety with  terminal singularities. 
Then there exists an at most countable family $(\Gamma_i)_{i \in I}$  of rational curves  on $X$
such that 
$$
0 < -K_X \cdot \Gamma_i \leq 2 \dim X - 1
$$
and
$$
\NEX = \NEX_{K_X \geq 0} + \sum_{i \in I} \R^+ [\Gamma_i].
$$
\end{theorem}

Given a projective variety $X$, the aim of the minimal model program is to replace $X$ by some birational
model such that either the canonical divisor is nef or we have a Mori fibre space structure. The way to get
closer to this model is to contract the extremal rays appearing in the cone theorem. The existence
of these contractions is assured by the contraction theorem: 

\begin{theorem} \label{theoremcontraction} 
Let $X$ be a normal $\Q$-factorial projective variety with  terminal singularities. 
Let $R$ be a $K_X$-negative extremal ray in $\NEX$.
Then there exists a morphism with connected fibres $\holom{\varphi}{X}{Y}$ onto 
a normal projective variety $Y$ such that a curve $C \subset X$
is contracted onto a point if and only if $[C] \in R$.

We call $\varphi$ the elementary Mori contraction associated to the extremal ray $R$.
\end{theorem}

The statements of Theorem \ref{theoremNE} and Theorem \ref{theoremcontraction} make sense 
in the more general setting of compact K\"ahler spaces, and one expects them to hold in arbitrary dimension.
However we will see in Section \ref{sectionMMP} that even for threefolds this requires a substantial 
amount of work. The first reason is that Kleiman's theorem \eqref{kleiman} is not true for arbitrary classes on non-algebraic 
K\"ahler manifolds, but in order for the MMP to exist it should hold for the canonical class. 
The second reason is hidden in the proof of Theorem \ref{theoremcontraction}: a morphism $\holom{\varphi}{X}{Y}$ between projective varieties is always defined by some globally generated line bundle, so the natural way 
to prove Theorem \ref{theoremcontraction} is to give sufficient conditions for a line bundle to be semiample.
This is achieved by the basepoint free theorem:

\begin{theorem} \label{theorembpf} 
Let $X$ be a normal $\Q$-factorial projective variety with  terminal singularities. 
Let $L$ be a nef line bundle such that $L-K_X$ is nef and big. Then $mL$ is generated by its global sections for all $m \gg 0$.
\end{theorem}

This very important technical result will not be of any help if we want 
to consider general K\"ahler spaces: the existence of a line bundle $L-K_X$ that
is nef and big implies that $X$ is Moishezon. Yet Moishezon spaces (with rational singularities) that
are K\"ahler are always projective \cite{Nam02}. Thus an important part of the work will be to find new ways to prove the contraction theorem.

While the cone and contraction theorem need completely new proofs in the K\"ahler setting, other parts of the minimal model
program can be easily generalised from the projective case: since we always contract $K_X$-negative extremal rays, the
contractions are {\em projective} morphisms polarised by $-K_X$. Moreover the existence of flips in dimension three
was proven by Mori \cite{Mor88} and Shokurov \cite{Sho92} in the setting of complex spaces, so we can use directly their results.

The notion of $\mathbb Q-$factoriality is important to run the MMP. 
For a non-projective complex space the canonical sheaf might not be a $\mathbb Q$-divisor,
so this requires a little care:

\begin{definition} Let $X$ be a normal compact complex space. 
We say that $X$ is $\mathbb Q$-factorial, if every Weil divisor is $\mathbb Q$-Cartier 
{\em and} there is a number $m \in \N$ such that the coherent sheaf $\sO_X(m K_X) = (\omega_X^{\otimes m})^{**}$
is locally free. 
\end{definition}

\section{K\"ahler spaces and the generalised Mori cone $\NAX$} \label{sectionkaehler}

In this section we introduce the cone $\overline{NA}(X)$ on a compact K\"ahler space $X$ which plays the role of the Mori cone 
$\overline{NE}(X)$ in the algebraic setting. 
The notion of a singular K\"ahler space was first introduced by Grauert \cite{Gra62}. 

\begin{definition} \label{definitionkaehler}
An irreducible and reduced complex space $X$ is K\"ahler if there exists a K\"ahler form $\omega$, i.e. a positive closed real $(1,1)$-form $\omega \in \mathcal A_\R^{1,1}(X)$ such that the following holds: for every point
$x \in X_{\sing}$ there exists an open neighbourhood $x \in U \subset X$ and a closed embedding $i_U: U \subset V$ into an open set $V \subset \mathbb C^N$, and  a strictly plurisubharmonic $C^{\infty}$-function $f : V \rightarrow \C$ with 
$ \omega|_{U \cap X_{\nons}} = (i \partial \overline \partial f)|_{U \cap X_{\nons}}$. 
\end{definition} 

For the notion and basic properties of $(p,q)-$forms and currents on singular spaces we refer to \cite{Dem85}. 
We will denote by $\sA^{p,q}_X$ the sheaf of $(p,q)$-forms, and by $\sD^{p,q}_X$ the sheaf of currents of bidegree $(p,q)$.
The relevant cohomology theory on K\"ahler spaces is the Bott-Chern cohomology, see 
 \cite[Defn. 4.6.2]{BEG13}:

\begin{definition} 
Let $X$ be a normal complex space. Let $\sH_X$ be the sheaf of real parts of holomorphic
functions multiplied with $i$. A $(1,1)$-form (resp. $(1,1)$-current) with local potentials on $X$ is a global section
of the quotient sheaf $\sA^{0,0}_X/\sH_X$ (resp. $\mathcal D^{0,0}_X/\sH_X$). Then the Bott-Chern cohomology is defined as 
$$
H^{1,1}_{\rm BC}(X) := H^1(X, \sH_X).
$$
\end{definition}

\begin{remark} 
An element of the Bott-Chern cohomology group can be viewed as a closed $(1,1)$-form with local potentials
modulo all the forms that are globally of the form $dd^c u$.  
Alternatively, we can also use 
 $(1,1)$-currents with local potentials to define Bott-Chern cohomology. 
 \end{remark}

To make the analogy to the projective case clearer, we define

\begin{definition} 
Let $X$ be a normal compact complex space in the Fujiki class $\mathcal C$. Then 
$$ N^1(X) := H^{1,1}_{\rm BC}(X).  $$ 
\end{definition}

The algebraic definition deals with classes of divisors; however in the non-algebraic setting there are too few divisors, so that this
space is too small to be useful. If $X$ has rational singularities - which will always be the case in our considerations - then 
$$ N^1(X) \subset H^2(X,\mathbb R).$$ 
Even in the algebraic case, the new space $N^1(X) $ can be larger than the traditional space $N^1(X) = NS(X) \otimes \mathbb R.$

If $X$ is projective, then the space $N_1(X)$ is the subspace of $H_2(X,\mathbb R)$ generated by classes of curves. Again we need a more general definition here.

\begin{definition} 
Let $X$ be a normal compact complex space in class $\mathcal C.$ We define $N_1(X) $ to be the vector space of real closed currents of bidimension $(1,1)$ modulo the
 following equivalence relation: 
 $ T_1 \equiv  T_2 $ if and only if
 $$ T_1(\eta) = T_2(\eta)$$
 for all real closed $(1,1)$-forms $\eta$ with local potentials.  
 \end{definition}
 
 In this setting, the analytic counterpart of the Mori cone $\overline{NE}(X)$  in the projective case is given by
 
 \begin{definition} Let $X$ be a normal compact complex space in class $\mathcal C.$ 
Then $\NAX \subset  N_1(X)$ is the closed cone generated by the classes of positive closed currents. 
The Mori cone is the closed subcone
 $$ \NEX \subset \NAX $$
generated by those positive closed currents arising as currents of integration over curves. 
 \end{definition} 
 
If $X$ is projective, $\NEX$ is just the usual Mori cone of curves. However, even if $X$ is a projective manifold, 
the cone $\NAX$ can be larger than $\NEX$, namely when $\rho(X) < h^{1,1}(X)$. 
 
 \begin{definition} Let $X$ be a normal compact complex space and $u \in H^{1,1}_{\rm BC}(X).$
 \begin{enumerate} 
 \item $u$ is pseudo-effective, if $u$ can be represented by a current $T \in \mathcal D^{1,1}(X)$ which is locally of the form  $\partial \overline{\partial} \varphi $
 with a plurisubharmonic function $\varphi.$ 
 \item Then $u$ is nef if $u$ can be represented by a form $\alpha$ with local potentials such that for some positive $(1,1)-$form $\omega $ on $X$ and
for every $\epsilon > 0$ there exists a $C^{\infty}-$function  $f_{\epsilon} $ such that
$$ \alpha + i \partial \overline{\partial} f_{\epsilon} \geq - \epsilon \omega.$$
\item  ${\rm Nef}(X) \subset N^1(X)$ is the cone generated by nef cohomology classes.
\end{enumerate}
\end{definition} 

The notion of nef divisors/classes is central for the MMP, so let us explain what the slightly technical definition above means.
For a K\"ahler space nef classes are limits of positive (i.e. K\"ahler) classes:

\begin{theorem}  \cite[Prop.6.1.iii)]{Dem92}
Let $X$ be a normal compact K\"ahler space. Then 
${\rm Nef}(X)$ is the closure of the K\"ahler cone 
\end{theorem} 

More geometrically we know that a $\Q$-Cartier divisor $u$ on a normal projective variety $X$
is nef if $u \cdot C \geq 0$ for all irreducible curves $C \subset X$.
On a non-algebraic K\"ahler space such a divisor can even be antinef!
Indeed if $X$ is a compact K\"ahler surface of algebraic dimension one, then
we have an elliptic fibration $\varphi: X \rightarrow B$ onto a projective curve $B$.
Set $u := - \varphi^* H$ with $H$ an ample divisor on $B$, so $u$ is antinef.
However all the curves $C \subset X$ are contracted by $\varphi$, so we have
$u \cdot C = 0$ for all curves $C \subset X$.
We will explain in Section \ref{sectionMMP} that this phenomena should not happen when $u = c_1(K_X)$.

In the projective setting, the nef cone (i.e., the closure of the ample cone) and the Mori cone $\NEX$ are dual. 
Here is the analogue in the K\"ahler case: given a real closed $(1,1)$-form $\omega$ on $X$ with local potentials, we define 
$$
\lambda_\omega \in N_1(X)^*, \ [T] \mapsto T(\omega).
$$
Notice that if  $T(\omega)=0$ for all  closed currents $T$ of bidimension $(1,1)$, then $\lambda_\omega=0$. Thus we obtain well-defined canonical map
$$
\Phi: N^1(X) \to N_1(X)^*, \ [\omega] \mapsto \lambda_\omega.
$$
 
 \begin{theorem} Let $X$ be a normal compact complex space in class $\mathcal C$ with only rational singularities. 
 \begin{enumerate}
\item $\Phi$ is an isomorphism.
\item If $\dim X = 3,$ then  $\Phi({\rm Nef}(X)) = \NAX.$
\end{enumerate} 
\end{theorem} 

The assumption on the dimension in the last statement should be superfluous.
Since the MMP is an iteration of morphisms and flips one has to be able to check that at every
step we remain in the category of K\"ahler spaces. Here is an example of such a criterion:

\begin{lemma} \label{lemmakaehler}
Let $X$ be a normal $\Q$-factorial compact K\"ahler threefold, and let $\holom{\varphi}{X}{X'}$ be a 
bimeromorphic Mori contraction. Let $\alpha$ be a nef class on $X$ such that the following holds:
$\alpha^3>0$, $\alpha^2 \cdot S = 0$ if and only if the surface $S$ is contracted
by $\varphi$ and $\alpha \cdot C = 0$ if and only if the curve $C$ is contracted by $\varphi$. Then 
$\alpha = \varphi^* \alpha'$ with $\alpha'$ a K\"ahler class on $X'$. In particular $X'$ is K\"ahler.
\end{lemma}

\begin{proof}
It is easy to see that $\alpha = \varphi^* \alpha'$ with $\alpha' \in N^1(X')$ (cf. \cite[Sect.3]{HP16}), the interesting part
is to show the K\"ahler property: since $\alpha^3>0$ the class $\alpha$ is nef and big \cite{DP04} 
and as a consequence of \cite{CT15, Bou04}, applied on a resolution of singularities, we can choose a K\"ahler current $T$ in $\alpha$ such that its Lelong level sets are in the $\varphi$-exceptional locus. The push-forward $\varphi_* (T)$ defines a K\"ahler current $T'$ in $\alpha'$ such that
its Lelong level sets are in the image of the $\varphi$-exceptional locus. By \cite[Prop.3.3]{DP04}, see also \cite{CT16}, we are done
if we show that the restriction of $\alpha'$ to all these sets is a K\"ahler class.
Yet $X$ has dimension three, so the image of the exceptional locus is either a union of points or a curve $C' \subset X'$.
Now observe the following: the map $\fibre{\varphi}{C'} \rightarrow C'$ is a projective morphism
over a projective variety, so it has a multisection $C$. By the projection formula $\alpha' \cdot C'$
is a positive multiple of $\alpha \cdot C$ which (by our assumption on intersections)
is positive. Thus $\alpha'|_{C'}$ is a K\"ahler class. 
\end{proof}
 
\begin{remark} In \cite{HP16} the K\"ahler property was shown by a more general theorem \cite[Thm.3.18]{HP16}
involving only conditions on $\NAX$. S.Boucksom pointed out that for the fact
in Step 2 of the proof that $T_{\infty}(E) < 0$, some additional arguments should be given. 
We know in the situation of Step 2 of Theorem 3.18  that 
$[ \mu_*(T_{\infty})] = 0.$ 
Now we may write 
$$ [T_{\infty} ] = [i_*(R)] $$
with a current $R$ supported on the exceptional divisor $E$. This is easily seen since there is a class $b \in H_2(E)$ such that $i_*(b)$  equals the class of
$[T_{\infty} ] $ in $H_2(\hat X).$ Decompose 
$ R = R^+ - R^- $
with positive closed currents $R^+$ and $R^-$ 
and define  a positive closed
$ \hat T = T_{\infty} + i_*(R^-).$
Notice
$ \mu_*(i_*(R^+)) = 0$ and 
let $\pi: \tilde X \to \hat X$ be a sequence of blow-ups such that $\tilde X$ is a compact K\"ahler manifold. 
Let $\tilde S \subset \tilde X$ and $\hat S \subset \hat X$ be the exceptional loci. 
Now  Step 1 of 3.18 is easily adapted to yield
$$  \chi_{\hat X \setminus \hat S} \hat T = 0. $$
Consequently, $T_{\infty} $ is supported on $\hat S,$ hence $T_{\infty}(E) < 0$.
 \end{remark} 
 
\section{MMP for K\"ahler threefolds} \label{sectionMMP}

\subsection{Existence of rational curves, bend-and-break} \label{subsectionrational}

In this whole section we denote by $X$ a $\Q$-factorial compact K\"ahler threefold with terminal singularities, for the outlines
of the proofs we will assume implicitly that $X$ is {\em smooth}. We also suppose:
\begin{equation} \label{hypothesis1}
\mbox{The canonical class $K_X$ is not nef.}
\end{equation}
One of the first fundamental contributions of Mori was to translate this numerical property into a geometric 
statement:

\begin{theorem} \cite{Mor79, Mor82}
Let $X$ be a projective manifold such that $K_X$ is not nef. Then there exists a {\em rational} curve $C \subset X$
such that $-K_X \cdot C>0$. 
\end{theorem} 

Mori's proof uses deformation theory of curves and a reduction to positive characteristic in an essential way, 
and for a long time it was not clear how to generalise his statement in a non-algebraic setting. The approach 
used in \cite{HP16, HP15} is to split the problem in two parts. The first part is due to M. Brunella:

\begin{theorem} \cite{Bru06} \label{theorembrunella}
Let $X$ be a $\Q$-factorial compact K\"ahler threefold with terminal singularities.
Then $K_X$ is not pseudoeffective if and only if $X$ is covered by
rational curves.
\end{theorem}

First observe that we may assume $X$ to be smooth: if $\hat X \to X$ is any desingularisation, then $K_{\hat X}$ is not pseudoeffective if and only if $K_X$ is not pseudoeffective, the
singularities of $X$ being terminal.
Now Brunella proves that the canonical bundle $K_\sF$ of a rank one foliation $\sF \subset T_X$ is not pseudoeffective if and only if the general leaf of $\sF$ is a rational curve. 
Since a non-algebraic K\"ahler threefold always admits a rank one foliation defined by a holomorphic two-form (by Kodaira's theorem), this implies the statement.
For the discussion of non-uniruled spaces we may therefore replace \eqref{hypothesis1} by 
\begin{equation} \label{hypothesis2}
\mbox{The canonical class $K_X$ is pseudoeffective, but not nef.}
\end{equation}
By the theorem of Demailly-P\v aun \cite{DP04} the existence of a cohomology class that is pseudoeffective, but not nef 
yields some first geometric information: there exists a proper subvariety $Z \subsetneq X$ such that $K_X|_Z$
is not pseudoeffective. If $X$ is smooth (or at least Gorenstein\footnote{
For simplicity of the exposition, we completely ignore the substantial 
difficulties in the non-Gorenstein case, cf. \cite[Sect.4.C, Sect.5]{HP16}.}) we can focus on the case where $Z$ is a surface (hence a divisor
in the threefold $X$). Indeed if $Z$ is a curve $C$, then $K_X|_C$ being not pseudoeffective is equivalent to $K_X \cdot C < 0$. 
The deformation theory of curves on threefolds now shows that 
the deformations of $C$ cover a surface $S \subset X$ \cite[II, Thm.1.15]{Kol96}.
Then the restriction $K_X|_S$ is not pseudoeffective since $S$ is covered by $K_X$-negative curves.

\begin{lemma} \cite[Lemma 4.1]{HP16} \label{lemmafundamental}
Let $X$ be a $\Q$-factorial compact K\"ahler threefold with terminal singularities that is not uniruled. Let $S \subset X$ be an irreducible
surface such that $K_X|_S$ is not pseudoeffective. Then $S$ is covered by rational curves.
\end{lemma}

\begin{proof}
Since $K_X$ is pseudoeffective we can use Boucksom's decomposition theorem \cite{Bou04}: there exist
irreducible surfaces $S_j \subset X$ such that
\begin{equation} \label{Bdecomposition}
K_X = \sum_{j=1}^r \lambda_j S_j + N(K_X),
\end{equation}
where $\lambda_j \geq 0$ and the cohomology class $N(K_X)$ is pseudoeffective and the restriction
of $N(K_X)$ to any surface in $X$ is pseudoeffective. One easily deduces that $S$ is one of
the surfaces $S_j$, so up to renumbering we can write 
$$
S = S_1 = \frac{1}{\lambda_1} K_X - \frac{1}{\lambda_1} (\sum_{j=2}^r \lambda_j S_j +  N(K_X)).
$$
Yet by adjunction this implies
$$
K_{S} = (K_X+S)|_{S} = \frac{\lambda_1+1}{\lambda_1} K_X|_S - \frac{1}{\lambda_1} (\sum_{j=2}^r \lambda_j (S_j \cap S) +  N(K_X)|_S).
$$
By hypothesis $K_X|_S$ is not pseudoeffective and $- \frac{1}{\lambda_1} (\sum_{j=2}^r \lambda_j (S_j \cap S) +  N(K_X)|_S)$ is anti-pseudoeffective. 
Thus $K_S$ is not pseudoeffective, hence $S$ is covered by rational curves (pass to a desingularisation whose canonical class is still not pseudoeffective). 
\end{proof}

Put together these arguments show the existence of $K_X$-negative rational curves if $K_X$ is not nef. 
The difference between this existence result and the classical cone Theorem \ref{theoremNE} is that
we have to show that if for some curve $C \subset X$ the degree $-K_X \cdot C$ is too large (say
$-K_X \cdot C > 2 \dim X -1$), then the cohomology class $C$ decomposes into $[C]= [C_1]+[C_2]$
with $C_i$ effective $1$-cycles. Mori proves this property via his famous bend-and-break technique,
again a reduction to positive characteristic plays an important role. 
In our case we can use \cite[Thm. 1.15]{Kol96} and reduce many arguments to considerations
in one of the surfaces $S_j$ appearing in the decomposition \eqref{Bdecomposition}.

\begin{proposition} \label{propositionNE} \cite[Thm.6.2]{HP16}
Let $X$ be a $\Q$-factorial compact K\"ahler threefold with terminal singularities that is not uniruled.
Then there exists a countable family $(\Gamma_i)_{i \in I}$  of rational curves  on $X$
such that 
$$
0 < -K_X \cdot \Gamma_i \leq 4
$$
and
$$
\NEX = \NEX_{K_X \geq 0} + \sum_{i \in I} \R^+ [\Gamma_i] 
$$
\end{proposition}

However, as observed in Section \ref{sectionkaehler}, this is not the statement that we want: the bimeromorphic
geometry of K\"ahler manifolds is governed by $\NAX$, not the much smaller $\NEX$. Somewhat surprisingly,
we have a precise analogue for $\NAX$:

\begin{theorem} \label{theoremNA}
Let $X$ be a $\Q$-factorial compact K\"ahler threefold with terminal singularities that is not uniruled.
Then there exists a countable family $(\Gamma_i)_{i \in I}$  of rational curves  on $X$
such that 
$$
0 < -K_X \cdot \Gamma_i \leq 4
$$
and
$$
\NAX = \NAX_{K_X \geq 0} + \sum_{i \in I} \R^+ [\Gamma_i] 
$$
\end{theorem}

\begin{proof} As before assume that $X$ is smooth.
The idea is to show that on the $K_X$-negative side, the cones $\NEX$ and $\NAX$ are quite similar, so the statement
reduces to Proposition \ref{propositionNE}. More precisely we know by
\cite[Cor.0.3]{DP04}\footnote{If $X$ is not smooth, all the computations are on a resolution
of singularities, cf. \cite[Sect.6]{HP16}. We expect that \cite[Cor.0.3]{DP04} holds for singular spaces
with mild singularities, see \cite{CT16} for some progress in this direction.}, that $\NAX$ is the closure of the convex cone generated by cohomology classes
of the form $[\omega]^2, [\omega] \cdot [S]$ and $[C]$ where $\omega$ is a K\"ahler form,
$S$ a surface and $C$ a curve on $X$. Note that \cite[Cor.0.3]{DP04} and \cite{CT16} are general results, valid in any dimension. 

Since $K_X$ is pseudoeffective we have $K_X \cdot [\omega]^2 \geq 0$, so these classes lie in $\NAX_{K_X \geq 0}$
and  are not of any interest for the statement. The case of a curve class $[C]$ is dealt with by Proposition \ref{propositionNE},
so the main problem is to understand the classes $[\omega] \cdot [S]$ which are $K_X$-negative. 
Now observe that
$$
0 > K_X \cdot [\omega] \cdot [S] = K_X|_S \cdot [\omega|_S].
$$
Since $\omega|_S$ is K\"ahler, the restriction $K_X|_S$ is not pseudoeffective, hence 
the surface $S$ is covered by rational curves by Lemma \ref{lemmafundamental}.
Yet this implies that (up to replacing $S$ by some resolution) that $H^2(S, \sO_S)=0$,
hence $S$ is projective and the cohomology class $[\omega|_S]$ is an ample $\R$-divisor.
Thus we can represent the class $[\omega] \cdot [S]$ by some curve class $[C]$ where $C \subset S$ is an
effective $\R$-divisor linearly equivalent to to the ample divisor $[\omega|_S]$. 
\end{proof}

\subsection{Contraction theorem - non-uniruled case} \label{subsectioncontraction}

Let $X$ be a $\Q$-factorial compact K\"ahler threefold with terminal singularities that is not uniruled.
We fix a $K_X$-negative extremal ray $R$ in the generalised Mori cone $\NAX$.
As a consequence of the cone theorem there exists a nef 
cohomology class $\alpha \in N^1(X)$ such that 
$$
\alpha^\perp \cap \NAX := \{
\gamma \in \NAX \ | \ \alpha \cdot \gamma = 0
\}  = R
$$
and $\alpha$ is adjoint, i.e. we can write $\alpha = K_X+\omega$ with $\omega$ a K\"ahler class.
If $X$ is a projective manifold one can choose $\alpha$ to be the class of a line bundle $L$ and the basepoint free theorem \ref{theorembpf} tells us that some positive multiple of $L$ is generated by its global sections.
The morphism defined by this multiple is then the contraction of the extremal ray $R$.

In the general setting of K\"ahler manifolds it is not clear whether $\alpha$ represents a line bundle,
in fact a morphism between non-algebraic spaces is almost never defined by a line bundle.
The only general tool for constructing morphisms in the analytic setting is given by the contraction
theorems of Grauert, Fujiki and Ancona-van Tan:

\begin{theorem} \label{theoremgrauert} \cite{AV84}
Let $X$ be a  complex space, and let $Z \subset X$ be a closed complex subspace.
Suppose that there exists a proper morphism $g: Z \rightarrow W$ such that the conormal sheaf
$I_Z/I_Z^2$ of $Z \subset X$ is ample on the $g$-fibres. Suppose also that the natural map
$$
g_* (\sO_Z/I_Z^{n+1}) \rightarrow  g_* (\sO_Z/I_Z^{n})
$$
is surjective for any $n \geq 1$.

Then there exists a morphism  $\holom{f}{X}{Y}$ such that $f|_{X \setminus Z}$ is an isomorphism onto its image and
$f|_Z=g$.
\end{theorem}

In our setting the natural candidate for the subspace $Z$ is the locus covered by
the curves in the extremal ray $R$. However from this description it is not clear how
to check the conditions of Theorem \ref{theoremgrauert}. Moreover the theorem does not
tell us how to check that $Y$ is a K\"ahler space. In order to overcome both of these difficulties
we work with the cohomology class $\alpha$.

Let us note first that since $\alpha=K_X+\omega$ is the sum of a pseudoeffective class and a K\"ahler class, it is big, so we have $\alpha^3>0$. Thus the null locus 
$$
\mbox{Null}(\alpha) := \bigcup_{V \subset X, \alpha^{\dim V} \cdot V=0} V
$$
is a countable union of proper subvarieties of $X$. By a theorem of Collins-Tosatti \cite{CT15} (we argue on a resolution in order to be able to apply \cite{CT15}), this locus
is actually equal to the non-K\"ahler locus of $\alpha$. Modulo some technical arguments one
can prove that $\mbox{Null}(\alpha)$ is exactly the locus covered by curves in the extremal ray
$R$ (in the algebraic case this follows a posteriori from the existence of the contraction).
Since we are in dimension three we obtain the following cases
\begin{itemize}
\item $\mbox{Null}(\alpha)$ is a finite union of curves;
\item $\mbox{Null}(\alpha)$ is an irreducible divisor $S$ and $\alpha|_S \neq 0$;
\item $\mbox{Null}(\alpha)$ is an irreducible divisor $S$ and $\alpha|_S = 0$.
\end{itemize}

In the first case we want to contract the curves via a small contraction onto points, in the second
case we want to contract the divisor $S$ onto a curve, and in the last case we want to contract
$S$ onto a point. We will deal with the first two cases, the last (and easiest) one is left to the interested
reader.

{\em Case 1: $\mbox{Null}(\alpha)$ is a finite union of curves.} In this case the morphism
$g$ in Theorem \ref{theoremgrauert} is simply the morphism that maps each connected component
of $\mbox{Null}(\alpha)$ onto a point. The difficult part is to check that the conormal sheaf is ample
on each of the fibres. However we know from \cite{Bou04} that there exists a modification
$\holom{\mu}{X'}{X}$ such that the image of the exceptional locus is $\mbox{Null}(\alpha)$
and 
$$
\mu^* \alpha = \omega' + E
$$
where $\omega'$ is a K\"ahler class and $E$ an effective divisor. Since $\alpha$ is numerically 
trivial on every curve $C \subset \mbox{Null}(\alpha)$, we obtain
$$
-E|_E = \omega'|_E. 
$$
Thus the conormal sheaf of $E$ is ample and we can contract the connected components of $E$
onto points by applying Theorem \ref{theoremgrauert} to $E \subset X'$. By construction this induces
the contraction on $X$.

{\em Case 2: $\mbox{Null}(\alpha)$ is an irreducible divisor $S$ and $\alpha|_S \neq 0$.}
This case is surprisingly difficult. We can restate the conditions in a more geometrical way:
the curves in the extremal ray $R$ cover the surface $S$, but it is not possible to connect
two arbitrary points in $S$ using curves in $R$. Thus we expect that the curves in the extremal
ray define a fibration $g: S \rightarrow C$ onto a curve $C$ and the extremal contraction 
will then contract the divisor $S$ onto this curve $C$. The problem is that in general the divisor $S$ can be non-normal, so standard techniques to define the fibration $g$ \cite{8authors} do not
apply. One therefore defines first a fibration $\tilde g: \tilde S \rightarrow \tilde C$ on the normalisation $\tilde S$ of $S$. 
If $X$ is terminal (and some other cases, cf. \cite[Ch.4.B]{CHP16}) a  computation of intersection numbers shows that $S$ is smooth or at least slc in a neighbourhood of a general fibre. 
An explicit construction then shows that $\tilde g$ descends to a fibration $g$.
Once we have defined the fibration $g$, the conditions on the conormal sheaf are easily verified since $S$ is a $\Q$-Cartier divisor.

We can now state the contraction theorem:

\begin{theorem} \label{thmcontr} 
Let $X$ be a $\Q$-factorial compact K\"ahler threefold with terminal singularities that is not uniruled,
and let $R \subset \NAX$ 
be a $K_X$-negative extremal ray. Then there exists a bimeromorphic morphism
$\holom{\varphi}{X}{Y}$ onto a normal compact K\"ahler space such that $-K_X$ 
is $\varphi$-ample (so $\varphi$ is a projective morphism) and $\varphi$ contracts exactly
the curves in the extremal ray $R$.
\end{theorem}

\begin{proof}
The considerations above establish the existence of a morphism $\holom{\varphi}{X}{Y}$ in the category of compact analytic
spaces. Thus  we are left to show
that $Y$ is also a K\"ahler space which can be done by checking the intersection properties
in Lemma \ref{lemmakaehler}. We have already seen that $\alpha^3>0$ and by the definition of $\alpha$ we have
$\alpha \cdot C=0$ for a curve $C$ if and only if $[C] \in R$. Hence the critical point is to show
that $\alpha^2 \cdot S=0$ if and only if $S$ is contracted by $\varphi$. Since $\alpha=K_X+\omega$ any
such surface has the property that $K_X|_S$ is not pseudoeffective, hence $S$ is covered by
rational curves by Lemma \ref{lemmafundamental}. One proves \cite[Prop.7.11]{HP16} that $S$ is actually
covered by $\alpha$-trivial rational curves, so $S$ is contracted by $\varphi$.
\end{proof}

\subsection{Running the MMP - non-uniruled case} 

We are now in position to establish Theorem \ref{theoremmain}. So let $X$ be a normal $\Q$-factorial compact   K\"ahler threefold with  terminal singularities and assume that
 $X$ is not uniruled. 
If $K_X$ is nef, there is nothing to prove. If $K_X$ is not nef, then by Theorem \ref{theoremNA}, there exists a $K_X$-negative extremal ray $R \subset \overline{NA}(X).$ By Theorem \ref{thmcontr}, the contraction 
$\varphi: X \to Y$ exists. Since $X$ is not uniruled, $\varphi$ is birational. If $\varphi$ is divisorial, then $Y$ is again a normal $\Q$-factorial compact   K\"ahler threefold with at most terminal singularities and we
continue the MMP with $Y$ replacing $X.$ If $\varphi$ is small, then by Mori's flip theorem \cite{Mor88}, the flip: $X^+ \to Y$ exists. Note here that the construction of the flip is analytically local around the 
exceptional locus of $\varphi.$ Again, $X^+$ is a  normal $\Q$-factorial compact   K\"ahler threefold with at most terminal singularities, and we may continue with $X^+$ instead of $X.$ 
Since the second Betti number drops by one in case of a divisorial contraction, it only remains to see that there is no infinite sequence of flips. This is however verified with exactly the same arguments
as in the projective case, see \cite{KMM87}.

\subsection{Uniruled threefolds}

Let $X$ be a $\Q$-factorial compact K\"ahler threefold with terminal singularities that is uniruled.
Then the canonical class $K_X$ is not pseudoeffective, so it seems that the arguments from the preceding sections
can not be used to establish the minimal model program. On the other hand the structure
of non-algebraic uniruled K\"ahler threefolds is very simple: let $X \dashrightarrow Y$
be the MRC-fibration. If $Y$ had dimension at most one, then $H^0(X, \Omega_X^2)=0$,
so $X$ is projective by Kodaira's criterion. Thus we see that $Y$ is a surface and the general
fibre $F$ is isomorphic to $\PP^1$. Choose now a K\"ahler class $\omega$
on $X$ such that $(K_X+\omega) \cdot F=0$. Then it follows from a result of P\v aun  \cite{Pau12} that
$K_X+\omega$ is pseudoeffective and we can develop the same theory as before, but
now we only contract extremal rays that are $K_X+\omega$-negative. Running this MMP
we obtain a bimeromorphic morphism $X \dashrightarrow X'$ such that
$K_{X'}+\omega'$ is nef. Such a class has very special properties:

\begin{theorem} \label{theoremmfs}
Let $X$ be a normal $\Q$-factorial compact K\"ahler threefold with  terminal singularities. Suppose that the base of the MRC-fibration $X \dashrightarrow Z$ has dimension two. Let $\omega$ be a K\"ahler class on $X$ such that $K_X+ \omega$ is nef
and $(K_X+\omega) \cdot F=0$ for $F$ a general fibre of the MRC-fibration.

Then there exists a {\em holomorphic} fibration $\holom{\varphi}{X}{S}$ 
onto a normal compact K\"ahler  surface $S$ such that $K_X+\omega=\varphi^* \omega_Y$
where $\omega_Y$ is a K\"ahler class on $S$.
\end{theorem}

This statement is an analytic analogue of the basepoint free theorem:
if $K_X+\omega$ is the class of a line bundle $L$, the statement says that $L$
is the pull-back of an ample divisor, in particular some multiple of $L$ is globally generated. 
The morphism $\varphi$ itself is in general not the contraction of an extremal ray,
but $-K_X$ is $\varphi$-ample, so $\varphi$ is a projective morphism. 
By \cite{Nak87} we can thus run a relative MMP over $S$ until we reach a Mori fibre space. Thus we may state

\begin{corollary} Let $X$ be a normal non-algebraic uniruled $\Q$-factorial compact   K\"ahler threefold with  terminal singularities. Then $X$ is a bimeromorphic 
via a sequence of contractions of extremal rays in $\overline{NA}$ and of flips to a Mori fiber space $\varphi:  X' \to S$. Here $X'$ is a normal $\Q$-factorial compact   K\"ahler threefold with  terminal singularities,
$S$ is a normal non-algebraic K\"ahler surface,
$-K_{X'}$ is $\varphi-$ample and $b_2(X') = b_2(S)  + 1.$ 
\end{corollary}

For the proof of Theorem \ref{theoremmfs} we use the nef reduction introduced in \cite{8authors}. This yields
an almost holomorphic map $\merom{\varphi}{X}{S}$ onto some surface $S$ and the goal is to show that
in our case this map is actually holomorphic. The key point of our argument is to show, via a second application
of P\v aun's theorem \cite{Pau12} that the numerical dimension of $K_X+\omega$ is two (hence equal to the ``nef dimension'' \cite{8authors}).

\section{Abundance for K\"ahler threefolds} \label{sectionabundance}

In Section \ref{sectionMMP} we explained that every K\"ahler threefold is bimeromorphic
either to a Mori fibre space or a minimal model (i.e. a K\"ahler space with nef canonical class).
In this section we will explain that such a minimal model is always good (i.e. the canonical class is semiample).

\begin{theorem} \label{abundance1} \cite{CHP16}
Let $X$ be a normal $\Q$-factorial compact K\"ahler threefold with terminal singularities
such that $K_X$ is nef. Then $K_X$ is semi-ample, that is some positive multiple $m K_X$ is globally generated.
\end{theorem} 

As in the algebraic case, the proof of the abundance theorem falls into two parts:

\begin{theorem} \label{abundance2} \cite{DP03}
Let $X$ be a normal $\Q$-factorial compact K\"ahler threefold with terminal singularities
such that $K_X$ is nef. Then we have $\kappa(X) \geq 0$. 
\end{theorem} 

and

\begin{theorem} \label{abundance3} \cite{CHP16}
Let $X$ be a normal $\Q$-factorial compact K\"ahler threefold with terminal singularities
such that $K_X$ is nef. If $\kappa (X) \geq 0$, then $K_X$ is semi-ample.
\end{theorem} 

\subsection{Existence of a section} \label{subsectionsection}
The existence of a global section for some pluricanonical divisor
is a very difficult problem which, even for projective threefolds, is done by several case distinctions
and ad-hoc arguments (cf. \cite{LP16} for some recent progress). 
For non-algebraic K\"ahler threefolds, Theorem \ref{abundance2}
was established in \cite{Pet01} with the exception that $X$ is simple (see Definition \ref{def:simple}) and not bimeromorphic to a quotient of a torus. As a consequence of the abundance theorem \ref{abundance1} we will prove in Section \ref{sectionapplications} 
that this exceptional case does not exist, but a priori one has to develop tools that do not rely on this kind
of classification result. In fact the paper \cite{Pet01} uses
heavily the classification of non-algebraic K\"ahler threefolds which are not simple (due to Fujiki \cite{Fuj83}), the remaining most difficult case was then settled in \cite{DP03}. 
A key ingredient is the Kawamata-Viehweg type vanishing:

\begin{theorem} \label{KV-Kaehler} \cite{DP03}
Let $X$ be a normal compact K\"ahler space of dimension $n$. 
Let $L$ be a nef line bundle on $X$ such that $c_1(L)^2 \ne 0.$ 
Then we have
$$ 
H^q(X,K_X + L) = 0
$$
for $q \geq n-1$. 
\end{theorem}

Actually, much more should be true: 

\begin{conjecture} \label{KVcon}
Let $X$ be a normal compact K\"ahler with canonical singularities, and let $L$ be a nef line bundle on $X$ of 
numerical dimension $\nu(L)$ (see Definition \ref{definitionnumerical}).  Then we have
$$ H^q(X,K_X + L) = 0, $$
provided $q > n- \nu.$ 
\end{conjecture}

In case $X$ is projective, this is the (generalised) Kawamata-Viehweg theorem \cite{Kaw82} or  \cite[6.13]{Dem01}. 
Notice that if $q = \dim X - 1,$ no assumption 
on the singularities is needed. Moreover, using the Grauert- Riemenschneider vanishing theorem
$$ R^q \pi_*(K_{\hat X}) = 0$$
for $q > 0$ and any desingularisation $\pi: \hat X \to X,$ it suffices to treat the smooth case. 
For new results towards Conjecture \ref{KVcon}, we refer to \cite{Cao14} and to the recent solution of 
Demailly's strong openness conjecture \cite{GZ15}. 

The second main ingredient is the inequality
$$
K_X \cdot c_2(X) \geq 0
$$
for a minimal simply connected K\"ahler threefold $X$ with algebraic
dimension $a(X) = 0$; i.e. $X$ does not carry a non-constant meromorphic function. We will recover this inequality by different arguments
in Section \ref{subsectionchern}, here we explain the argument from \cite{DP03} which is of 
independent interest: philosophically speaking, this inequality comes from
Enoki's theorem that the tangent sheaf of $X$ is $K_X$-semi-stable
when $K_X^2 \ne 0$ resp.\ $(K_X,\omega)$-semi-stable when $K_X^2 = 0$ (see Definition \ref{def:ss})
where $\omega$ is any K\"ahler form on $X$. Now if this semi-stability
with respect to a degenerate polarization would yield a Miyaoka-Yau
inequality, then $K_X \cdot c_2(X) \geq 0$ would follow. However this
type of Miyaoka-Yau inequalities with respect to degenerate polarizations is
not known. In the projective case, the inequality is deduced from
Miyaoka's generic nefness theorem (which uses char. $p$-methods).
Instead the paper \cite{DP03} approximates $K_X$ (in cohomology) by K\"ahler forms
$\omega_j$. If $T_X$ is still $\omega_j$-semi-stable for sufficiently
large $j$, then the usual Miyaoka-Yau inequality can be applied and in the limit 
the inequality $K_X \cdot c_2(X) \geq 0$ is established. Otherwise one examines
the maximal destabilizing subsheaf which essentially (because of $a(X)
= 0$) is independent of the polarization.

\smallskip \noindent
The third main ingredient is to prove the boundedness $h^2(X,mK_X) \leq 1$
in the case $\nu(X)=1$ (the case $\nu(X)=2$ is Theorem \ref{KV-Kaehler}).
This boundedness is shown under the additional assumption that
$a(X) = 0$ and that $\pi_1(X) $ is finite (otherwise by a result of
Campana $X$ is already bimeromorphic to a quotient of a torus). The main point is that if $h^2(X,mK_X)
\geq 2$, then we obtain ``many'' non-split extensions
$$ 0 \to K_X \to \mathcal E \to \ mK_X \to 0$$
and we analyze whether $ \sE$ is semi-stable or not. The assumption on
$\pi_1$ is used to conclude that if $\mathcal E$ is projectively flat, then
$\mathcal E$ is trivial after a finite \'etale cover.
\smallskip 

\noindent
From these three ingredients, Theorem \ref{abundance2}  follows by
applying Riemann-Roch calculations on a desingularisation of $X$.

\subsection{Semi-ampleness: The general strategy} 
For the discussion of Theorem \ref{abundance3},
we recall the definition of the numerical dimension:

\begin{definition} \label{definitionnumerical}
Let  $X$ be a normal compact K\"ahler space, and let $L$  be a nef line bundle on $X$.
Then the numerical dimension of $L$ is given by 
$$
\nu(X) := \max \{ m \in \N \ | \  c_1(L)^{m} \not \equiv 0 \}.
$$
If $K_X$ is $\mathbb Q$-Cartier of index $d$, we set
$$ 
\nu(X) := \nu(d K_X) 
$$
for the numerical dimension of $X$. 
\end{definition} 
It is easy to see that
$$ 
\kappa (L) \leq \nu(L),
$$
and if $L$ is semi-ample, then equality holds. 
Kawamata \cite{Kaw85b} discovered that if $L$ is the canonical class, then the converse holds:

\begin{theorem} 
Let $X$ be a normal compact K\"ahler space with only canonical singularities. Assume that $K_X$ is nef
and that $\kappa (X) = \nu(X).$ Then $K_X$ is semi-ample.
\end{theorem} 

The paper \cite{Kaw85b} deals with the algebraic case, but the methods also work
for K\"ahler spaces, see also \cite{Nak87} and \cite{Fuj11}. 

If $\nu(X) = \dim X$, then $K_X$ is big. Since a K\"ahler Moishezon space with rational singularities is projective \cite{Nam02}, we can apply the base point free theorem to see that $K_X$ is semiample.
If $\nu(X) = 0,$ and if we assume $\kappa (X) \geq 0,$ then $mK_X = \sO_X$ for some positive number $m \in \N$. 
The potential case that $\kappa (X) = 1$ and $\nu(X) = 2$ is ruled out exactly as in the algebraic case, \cite[7.3]{Kaw85} 

\begin{conclusion} In order to prove Theorem \ref{abundance3}, it suffices to rule out the cases
$$ 
\kappa (X) = 0, \  \nu(X) = 1,2. 
$$
\end{conclusion} 

Since the Kodaira dimension is non-negative, there is a positive number $m \in \N$ and an effective divisor
$$ D \in \vert mK_X \vert.$$ 
The standard method to prove that $\kappa(X) \geq 1$ is to consider the restriction map
$$
r : H^0(X, d(m+1)K_X) \rightarrow H^0(D, d K_D)
$$
for a suitable positive integer $d.$ 
 Arguing by induction on the dimension we aim to prove that $H^0(D, d K_D) \neq 0$ for some $d \in \N$ and that 
 some non-zero section $u \in H^0(D, d K_D)$ lifts via $r$ to a global section $\tilde u \in H^0(X, d(m+1)K_X)$ on $X$.
However, $D$ might be very singular and therefore it is not possible to analyse the divisor $D$ directly. 
In order to circumvent this difficulty, Kawamata \cite{Kaw92c} developed the strategy, further explored in \cite{Uta92}, 
to consider log pairs $(X, B)$ with $B = \supp D$ and to improve the singularities of this pair
via certain birational transformations. 
This requires deep techniques of birational geometry of pairs within the theory of minimal models. In particular we have to run a log MMP for certain log pairs $(X, \Delta)$,
which can be stated as follows. 

\subsection{Semi-ampleness: Use of a log-MMP} 

\begin{theorem} \label{theoremdltMMP} \cite{CHP16}
Let $X$ be a normal $\Q$-factorial compact K\"ahler threefold which is not uniruled. 
Let $D \in |m K_X|$ be a pluricanonical divisor and set $B:=\supp D$. Suppose that the pair $(X, B)$ is dlt.
Then there exists a terminating $(K+B)$-MMP, that is, there exists a bimeromorphic map
$$
\varphi: (X, B) \dashrightarrow (X', B':=\varphi_*B)
$$
which is a composition of $K+B$-negative divisorial contractions and flips such that $X'$ is a 
normal $\Q$-factorial compact K\"ahler threefold, the pair $(X', B')$ is dlt and $K_{X'}+B'$ is nef.
\end{theorem}

 This is of course not the general log-MMP  for K\"ahler threefolds, but it is sufficient for our purposes. The main obstacle for a general theorem is to show the existence of contractions for dlt pairs $(X, B)$
where $B$ has non-integer coefficients. 

We will not explain the proof of this result \cite[Sect.3, Sect.4]{CHP16} here, it requires several improvements of the arguments we have seen in Section \ref{sectionMMP}. Let us rather consider its consequences:
we consider the most complicated case $\nu(X) = 2$. Our goal is to prove  that $\kappa (X) \geq 1$.
Using the MMP of Theorem \ref{theoremdltMMP},
this can be reduced by very technical arguments to the following:

\begin{proposition} \label{reduction} 
Let $X$ be a normal $\Q$-factorial compact K\"ahler threefold with at most klt singularities. Suppose that 
there exists
a divisor $D \in |m K_{X}|$ with the following properties:
\begin{enumerate}
\item Set $B:=\supp D$. The pair $(X, B)$ is lc and $X \setminus B$ has terminal singularities.
\item The divisor $K_{X}+B$ is nef and we have $\nu(K_{X}+B)=2$. Moreover we have
$
\kappa(X) = \kappa(K_X+B).
$
\item For every irreducible component $T \subset B$ we have $(K_{X}+B)|_{T} \neq 0$.
\item We have $(K_{X}+B) \cdot K_{X}^2 \geq 0$.
\end{enumerate}
Then $\kappa(X) \geq 1$.
\end{proposition}

The acroynym klt stands for "Kawamata log terminal", and the definition is the same as in the algebraic setting. Properly speaking, $X$ is klt if there
is a positive integer $m$ such that $(\omega_X^{\otimes m})^{**}$ is locally free and there is
a resolution $\pi: \tilde X \to X$ such that 
$$\omega^{\otimes m}_{\tilde X} = \pi^* ((\omega_X^{\otimes m})^{**}) \otimes \mathcal O_{\tilde X}(\sum a_imE_i),$$
where the $E_i$ are the $\pi-$exceptional divisors and all $a_i > -1.$ 
In other words, we have the usual equation
$$ K_{\tilde X} =  \pi^*(K_X) + \sum a_i E_i $$
with $a_i > -1.$ 

We will explain the proof of this proposition in Subsection \ref{subsectiontheproof}, but this requires
another ingredient:

\subsection{Semi-ampleness: Generic nefness and Chern class inequalities} \label{subsectionchern}
We saw in Subsection \ref{subsectionsection} that a Chern class inequality in the spirit of \cite[6.1]{Miy87} plays a crucial role for nonvanishing, this problem appears again for abundance.
Having in mind our application, we restrict ourselves to threefolds; moreover we shall need -
and this is crucial - only to consider threefolds with isolated singularities. 
We do not make any attempt to define Chern classes of coherent sheaves on singular spaces, all that we need 
are the intersection numbers 
$$ \alpha \cdot c_1(\sF)^2,\  \alpha \cdot c_2(\sF), \  \alpha^2 \cdot c_1(\sF).$$
where $\alpha  \in N^1(X).$ 
For example, let us define $\alpha \cdot c_2(\sF).$ Choose a desingularisation $\pi: \hat X \to X;$ then $c_2(\pi^*(\sF)) \in H^4(\hat X, \bZ)$
is well-defined, and we
set
$$ \alpha \cdot c_2(\sF) = \pi^*(\alpha) \cdot c_2(\pi^*(\sF)).$$
See \cite[7.1]{CHP16} for details. 
In particular, slope stability with respect to a nef class can be defined:

\begin{definition} \label{def:ss} 
Let $X$ be a normal compact K\"ahler threefold with isolated singularities 
and let $\alpha$ be a nef class on $X$. We say that a non-zero torsion-free sheaf $\sF$ is $\alpha$-semistable (resp. $\alpha$-stable)
if for every non-zero saturated subsheaf $\sE \subset \sF$ we have
$$
\mu_{\alpha}(\sE) := \frac{\alpha^{2} \cdot c_1(\sE) }{\rk \sE} \leq \frac{\alpha^{2} \cdot c_1(\sF) }{\rk \sF}  =: \mu_{\alpha}(\sF) \qquad (\mbox{resp.} \ <). 
$$
\end{definition}

The generic nefness notion we shall use is given in 

\begin{definition} \label{definitiongenericallynef}
Let $X$ be a normal compact K\"ahler threefold,
and let $\alpha$ be a nef class on $X$. A non-zero torsion-free coherent sheaf $\sF$ on $X$
is $\alpha$-generically nef if for every torsion-free quotient sheaf
$\sF \rightarrow Q \rightarrow 0$ we have
$$
\alpha^{2} \cdot c_1(Q) \geq 0.
$$
\end{definition}

In this setting the following Chern class inequality holds. 

\begin{theorem} \label{theoremnonnegative}
Let $(X, \omega)$ be a compact K\"ahler threefold with isolated singularities.  Let $\sF$ be a non-zero reflexive coherent sheaf
on $X$ such that $\det \sF$ is $\Q$-Cartier. Suppose that there exists a pseudoeffective class $P \in N^1(X)$ such that
$$
L := c_1(\sF) + P
$$
is a nef class. Suppose furthermore that for all $0 < \varepsilon \ll 1$ the sheaf $\sF$ is $(L+\varepsilon \omega)$-generically nef. Then we have
$$
L \cdot c_2(\sF) \geq \frac{1}{2} (L \cdot c_1^2(\sF) - L^3).
$$
In particular, if $L \cdot c_1^2(\sF) \geq 0$ and $L^3=0$, then 
\begin{equation} \label{nonnegative}
L \cdot c_2(\sF) \geq 0.
\end{equation}
\end{theorem}

The proof is based on a Bogomolov inequality for stable sheaves on K\"ahler threefolds with isolated singularities which we state next.

\begin{theorem} \label{theorembogomolov}
Let $X$ be a normal compact K\"ahler threefold with isolated singularities,
and let $\alpha$ be a K\"ahler class on $X$.  Let $\sF$ be an $\alpha-$stable non-zero torsion-free coherent sheaf on $X$.
Then we have
$$
\alpha \cdot c_2(\sF)  \geq \big(\frac{r-1}{2r}\big)  \alpha \cdot  c_1^2(\sF).
$$ 
\end{theorem}

Theorem \ref{theoremnonnegative} can be applied to $(\Omega_X)^{**}$ (or $\Omega_X/{\rm torsion}$) since we have
the following

\begin{proposition} \label{propositiongenericallynef}
Let $X$ be a normal compact K\"ahler space of dimension $n$ with canonical singularities. 
Suppose that $K_X$ is nef or $\kappa(X) \geq 0$.
Then $\Omega_X /{\rm torsion}$ is generically nef with respect to any nef class $\alpha$, i.e. for every torsion-free quotient sheaf
$$
\Omega_X \rightarrow \mathcal Q \rightarrow 0,
$$
we have $\alpha^{n-1} \cdot c_1(\mathcal Q) \geq 0$.
\end{proposition} 

This result is a K\"ahler version of Miyaoka's generic nefness theorem which makes the weaker
assumption that $K_X$ is pseudoeffective. Our proof is by reduction to Enoki's theorem \cite{Eno88}, which makes use of the solution of a Monge-Amp\`ere equation. 

\subsection{Semi-ampleness: proof of Proposition \ref{reduction}} \label{subsectiontheproof}
Running some relative MMP one sees that there exists a terminal modification
$$
\holom{\mu}{X'}{X}
$$ 
of $X,$ 
i.e., $X'$ has only terminal singularities, and 
there exists an effective $\Q$-divisor $\Delta$ such that
$$
K_{X'} + \Delta \sim_\Q \mu^* K_X.
$$ 
Choose $m > 0$ such that  $m(K_{X}+B) =: L$ is Cartier and set $L':= \mu^* (L)$.
The key is now the basic Chern class inequality 
\begin{equation} \label{todd}
 L' \cdot (K_{X'}^2+c_2(X'))\geq 0,
\end{equation}
which comes down to proving that $L' \cdot c_2(X') \geq 0.$ 
Now
by Proposition \ref{propositiongenericallynef}, the sheaf $\Omega_{X'}$ is generically nef.
Since
$$
K_{X'} + \Delta + \mu^* B \sim_\Q \mu^* (K_X+B)
$$ 
is nef, the conditions of Theorem \ref{theoremnonnegative} are satisfied with  $\sF:=(\Omega_{X'})^{**}$ and $P:= \Delta + \mu^* B$.
Thus we conclude 
$$ L'\cdot c_2(X') \geq 0,$$
hence the Chern class inequality (\ref{todd}) is established.
Now by a Riemann-Roch formula,
$$
\chi(X', \sO_{X'}(nL')) = n L'\cdot \frac{K_{X'}^2+c_2(X')}{12} + \chi(X', \sO_{X'})
$$
for all $n \in \N$.
Thus \eqref{todd} yields a constant $k$ such that 
$$
\chi(X', \sO_{X'}(nL')) \geq  k 
$$
for all $n \in \N$,
and therefore 
$$
\chi(X, \sO_X(nL)) \geq  k 
$$
for all $n \in \N$. 
Now we show that $H^3(X,\sO_X(nL)) = 0$ and that $h^2(X,\sO_X(nL))$ is constant for large $n,$ hence 
\begin{equation} \label{hzero}  
h^0(X, nL) \geq h^1(X, nL) + c
\end{equation}
with some constant $c \in \Z$. Now comes the inductive argument we have been aiming for: 
$B$ is a minimal surface with at most slc singularities and $\nu(B)=1$,
so we know by abundance that $h^0(B, \mathcal O_B(nL))$ grows linearly. 
The Euler characteristic $\chi(B, \mathcal O_B(nL))$ is constant, moreover one can show 
$$ h^2(B,\mathcal O_B(nL)) = 0$$
for $n \gg 0$. Thus $h^1(B, \mathcal O_B(nL))$ also grows linearly.
Using the restriction sequence (cf. \cite[Sect.8]{CHP16} for details) one deduces that $h^1(X,\mathcal O_X(nL)) $ grows at least linearly.  
Therefore  we have
$$
\kappa(X) = \kappa(K_X+B) = \kappa (L) \geq 1.
$$

\section{Applications} \label{sectionapplications}

\subsection{Simple K\"ahler spaces}
One of the initial motivations for the study of the K\"ahler MMP was to describe 
the ``most non-algebraic" K\"ahler spaces, i.e. those that contain only a few subvarieties.

\begin{definition} \label{def:simple}  
Let $X$ be a normal compact K\"ahler space. 
We say that $X$ is simple, if there is no proper positive-dimensional subvariety 
through a very general point of $X$.
\end{definition} 

Examples of simple K\"ahler spaces are general tori or general hyperk\"ahler manifolds (say the 
Hilbert scheme $S^{[n]}$ of a K3 surface without curves). It is expected that all simple spaces 
arise from these two types of manifolds by standard constructions, i.e bimeromorphic transformations and finite maps. 
The MMP for K\"ahler threefolds confirms this in dimension three:

\begin{theorem} \label{theoremkummer} Let $X$ be a normal 
$\Q$-factorial compact K\"ahler threefold with terminal singularities. Suppose that $X$ is simple.
Then there exists a bimeromorphic morphism $X \rightarrow T/G$ where $T$ is a torus
and $G$ a finite group acting on $T$. 
\end{theorem} 

\begin{proof}
Since $X$ is not uniruled, we know by Theorem \ref{theoremmain} that there exists a minimal model $X \dashrightarrow X'$. Since $X$ (and hence $X'$) is not covered by divisors we see that $\kappa(X')=0$. 
Since $K_{X'}$ is semiample by Theorem \ref{theoremmain} we have $\nu(X')=\kappa(X')=0$, i.e. the canonical divisor $K_{X'}$  is numerically trivial. Since $X$ (and hence $X'$) is not covered by curves, the structure theorem \ref{st}, stated below, yields that 
$$X' \simeq T/G$$ 
with $T$ a torus and $G$ a finite group. Since $X$ (and hence $X'$) is not covered by positive-dimensional subvarieties,
the torus $T$ has no positive-dimensional subvarieties. In particular $T/G$ has no positive-dimensional
subvarieties, so $X \dashrightarrow T/G$ extends to a morphism.
\end{proof} 

\begin{theorem}  \label{st}
Let $X$ be a non-algebraic $\Q$-factorial compact K\"ahler threefold with terminal singularities.
If $K_X \equiv 0$, there exists a Galois cover $f: \tilde X \to X$ that is \'etale in codimension one 
such that either $X$ is a torus or a product of an elliptic curve and a K3 surface. 
\end{theorem}

This result should be understood as a generalisation of the Beauville-Bogomolov decomposition to
{\em singular} K\"ahler threefolds, the proof is based on the observation that after a cyclic covering
one has $\chi(X, \sO_X)=0$ (cf. \cite{CHP16} for details and further structure results for non-algebraic K\"ahler threefolds).
We refer to \cite{Cam04c}, \cite{GKP11}, \cite{Dru16} 
for further information concerning the Beauville-Bogomolov decomposition.

\subsection{Algebraic approximation of K\"ahler spaces}

\begin{definition} Let $X$ be a normal compact K\"ahler space. We say that $X$ is algebraically approximable
if there exists a proper flat morphism $\pi: \mathcal X \to B$ 
to a complex space $B$, such that the following holds.
\begin{enumerate} 
\item there is a point $0 \in B $ such that $X_0 \simeq X;$
\item there is a sequence $(t_k)_{k \in \N}$ in $B$, converging to $0$ such that the fibre $X_{t_k}$ is a projective
variety for all $k \in \N$.
\end{enumerate} 
\end{definition} 

Originally, Kodaira asked whether any compact K\"ahler manifold can be approximated algebraically. However, C. Voisin 
constructed counterexamples in dimension at least $4$ \cite{Voi04}. In dimension at least $10$, more is true: there are compact K\"ahler manifolds $X$ such that  no
smooth bimeromorphic model can be approximated algebraically \cite{Voi06}. In view of the MMP, it is natural to consider minimal models of $X$, which are usually singular. 
So one can formulate:

\begin{conjecture} (Modified Kodaira Conjecture) \\
Let $X$ be a compact K\"ahler manifold. Then there exists a bimeromorphic model  $X \dashrightarrow X'$
where $X'$ is normal $\Q$-factorial compact K\"ahler space with terminal singularities
that is algebraically approximable. 
\end{conjecture} 

In fact if $X$ is not uniruled some minimal model of $X$ should be algebraically approximable.
A first partial positive answer in dimension $3$ is given by P.Graf:

\begin{theorem}  \cite{Gr16}
Let $X$ be a smooth compact K\"ahler threefold with $\kappa (X) = 0$. 
Then $X$ has a minimal model which is algebraically approximable.
\end{theorem} 

Actually, the deformation Graf constructs is even locally trivial (in the sense that it preserves the
singularities \cite[Defn.2.10]{Gr16}).

\section{Outlook}

In this final section we would like to explain some of the challenges that will appear in the construction
of the minimal model program for K\"ahler spaces of arbitrary dimension.

{\em 1. Existence of rational curves.} In Subsection \ref{subsectionrational} we saw
that the construction of rational curves for (not necessarily algebraic) K\"ahler threefolds
is done in two steps: first Brunella's theorem on rank one foliations describes the case where the canonical class is not pseudo-effective, then the deformation theory of curves on threefolds allows to describe $K_X$-negative curves when $K_X$  is pseudoeffective. In higher dimension it is possible to replace the deformation theoretic part by a subadjunction argument \cite{CH15} 
that reduces the problem to the characterisation of uniruledness in lower dimension. More precisely the existence
of rational curves for K\"ahler manifolds is reduced to the following difficult 

\begin{conjecture} \label{conjectureBDPP}
Let $X$ be a compact K\"ahler manifold. Then the canonical class $K_X$ is pseudoeffective if and only if $X$ is not uniruled.
\end{conjecture}

Of course, if $X$ is uniruled, then $K_X$ cannot be pseudo-effective. In case $X$ is projective, the conjecture is a theorem of Miyaoka-Mori, \cite{MM86}; again the construction relies on char $p$ methods. 
Here completely new methods seem to be necessary. One step could be to establish the K\"ahler version of the main result in \cite{BDPP13}. 

{\em 2. Contracting extremal rays.} The nef supporting class $\alpha$ introduced in Subsection \ref{subsectioncontraction}
is an extremely useful tool that mimics the role of the line bundle $L$ appearing
in the basepoint free theorem \ref{theorembpf}. However all the considerations aim at verifying
the conditions of the Grauert-Fujiki-Ancona-Van Tan contraction theorem \ref{theoremgrauert} which can be quite lengthy.
In particular exhibiting the proper map on the exceptional locus/Null locus becomes even more difficult 
since even a posteriori exceptional loci of Mori contractions can be very singular. 
It seems that new ideas are necessary to deal with this problems, a good start would be a new version of Grauert's theorem
with conditions that are more adapted to the Mori setting.

{\em 3. Checking the K\"ahler property.}
One of the surprising features of the MMP for K\"ahler threefolds is that a posteriori the positivity properties
of all the divisors (and cohomology classes) appearing in our proofs can be checked by intersecting with curves.
The following example shows that already in dimension 4 this is no longer true:

\begin{example}
Let $B$ be a K3 surface that does not contain any curves (in particular $B$ is K\"ahler, but not projective). Denote by $X'$ the projectivisation of the rank three vector bundle $\Omega_B \oplus \sO_B$, then $\psi: X' \rightarrow B$ is a $\PP^2$-bundle so $X'$ is a compact K\"ahler fourfold.
Denote by $S' \subset X'$ the $\psi$-section corresponding to the quotient line bundle
$\Omega_B \oplus \sO_B \rightarrow \sO_B$, and let $\holom{\mu}{X}{X'}$ be the blow-up
of $X'$ along the surface $S'$. Then $X$ is a compact K\"ahler fourfold, and $\mu$ is the contraction of a $K_X$-negative extremal ray in $\NAX$ (in particular $-K_X$ is relatively ample). Note that there does not exist a surface $S \subset X$ that dominates $S' \subset X$: indeed the $\mu$-exceptional divisor $E \rightarrow S'$ is isomorphic to $\PP(\Omega_B) \rightarrow B$ and it is easy to check that the $\PP^1$-bundle $\PP(\Omega_B) \rightarrow B$ does not admit any meromorphic multisections.

Challenge: let $\alpha$ be a nef supporting class (cf. Subsection \ref{subsectioncontraction}) of the extremal ray contracted by $\mu$. Show that we can choose $\alpha$ such that $\alpha = \mu^* \omega$ with $\omega$ a {\em K\"ahler} class on $X'$. 
\end{example}

\bibliographystyle{amsplain}

\newcommand{\etalchar}[1]{$^{#1}$}
\def\cprime{$'$}

\end{document}